\documentclass[12pt,letterpaper]{article}
\usepackage{amssymb, amsmath, amsthm, esint}
\numberwithin{equation}{section}
\theoremstyle{plain}
\newtheorem{theorem}[equation]{Theorem}
\newtheorem{lemma}[equation]{Lemma}
\newtheorem{proposition}[equation]{Proposition}
\newtheorem{corollary}[equation]{Corollary}
\theoremstyle{definition}
\newtheorem{definition}[equation]{Definition}
\newtheorem{example}[equation]{Example}
\newtheorem{remark}[equation]{Remark}
\newcommand*{\NN}{\mathbb{N}}
\newcommand*{\ZZ}{\mathbb{Z}}
\newcommand*{\RR}{\mathbb{R}}
\newcommand*{\R}{\mathbb{R}}
\newcommand*{\bcal}{{\mathcal B}}
\newcommand*{\ccal}{{\mathcal C}}
\newcommand*{\hcal}{\mathcal{H}}
\newcommand*{\wcal}{\mathcal{W}}
\newcommand*{\eps}{\varepsilon}
\newcommand*{\Om}{\Omega}
\renewcommand*{\phi}{\varphi}
\newcommand*{\Be}{B_{1,1}}
\newcommand*{\dOm}{\partial\Omega}
\newcommand*{\dd}{{\mathrm d}}
\newcommand*{\loc}{{\mathrm{loc}}}
\DeclareMathOperator{\dist}{dist}
\DeclareMathOperator{\rad}{rad}
\DeclareMathOperator{\Lip}{Lip}
\DeclareMathOperator{\LIP}{LIP}
\DeclareMathOperator{\diam}{diam}
\DeclareMathOperator{\Ext}{Ext}
\begin{document}
\title{Trace and extension theorems for functions of bounded variation
\thanks{%
{\it Keywords\,}: Besov, BV, metric measure space,
co-dimension $1$ Hausdorff measure, trace, extension, Whitney cover.}}
\author{Luk\'{a}\v{s} Mal\'{y}\and Nageswari Shanmugalingam\and Marie Snipes}
\date{July 6, 2016}
\maketitle
\begin{abstract}
In this paper we show that every $L^1$-integrable function on $\dOm$ can be
obtained as the trace of a function of bounded variation in $\Om$ whenever
$\Om$ is a domain with regular boundary $\dOm$ in a doubling metric measure
space. In particular, the trace class of $BV(\Om)$ is $L^1(\dOm)$ provided
that $\Om$ supports a $1$-Poincar\'e inequality. We also construct a bounded
 linear extension from a Besov class of functions on $\dOm$ to $BV(\Om)$.
\end{abstract}
\section{Overview}
In Dirichlet boundary value problems in analysis, one prescribes the trace
value of the solution at the boundary of the domain. Given a domain $\Omega$,
it is therefore natural to ask what class of functions on the boundary can be
realized as the traces of functions of specified regularity on the domain.

The model problem that motivates our study is the problem of finding least
gradient functions from the class of functions of bounded variation (BV),
with prescribed boundary data, see \cite{AnzGia,Giu, M-Sr, SWZ}. Therefore the
regularity of the extended function inside the domain is BV regularity.

The paper~\cite{AnzGia} first studied the trace and extension problem for
functions of bounded variation in Euclidean Lipschitz domains. It was shown
there that the trace functions of BV functions on the domain lie in the
$L^1$-class of the boundary. In contrast, the work~\cite{Gag} demonstrated that
every $L^1$-function on the boundary of a Euclidean half-space (and hence
boundaries of Lipschitz domains) has a BV extension to the half-space.
Together, these two results indicate that the trace class of BV functions on
a Euclidean Lipschitz domain is the $L^1$-class of its boundary.

In the metric setting, a version of the Dirichlet problem associated with BV
functions was considered in~\cite{GMS1, GMS2, GMS3, HaKiLa}, but their notion
of trace required that the BV function be defined on a larger domain.
In~\cite{LS} this requirement was dispensed with for domains whose boundaries
are more regular (Euclidean Lipschitz domains satisfy this regularity
condition). In~\cite{LS} it was shown that if in addition the domain supports a
$1$-Poincar\'e inequality, then the trace of a BV function on the domain lies
in a suitable $L^1$-class of the boundary, thus providing an analog of the
results of~\cite{AnzGia} in the metric setting. The recent work~\cite{Vit} gave
an analog of the extension result of~\cite{Gag} for Lipschitz domains in
Carnot--Carath\'eodory spaces, which indicated that it is possible to identify
the trace class of BV functions in more general metric measure spaces. The goal
of this paper is to provide such an identification, by adapting the technique
of~\cite{Gag} to the metric setting.

In this paper $\Om$ denotes a bounded domain in a metric measure space
$(X,d,\mu)$. The natural measure on $\Om$ is the restriction of $\mu$ to $\Om$.
The measure we consider on the boundary $\dOm$ is the co-dimension $1$
Hausdorff measure $\mathcal{H}:=\mathcal{H}\vert_{\dOm}$
(see~\eqref{eq:deff-mathcal} below). The function spaces related to $\dOm$ will
have norms computed using the measure $\mathcal{H}$, and this being understood,
we will not explicitly mention the measure in the notation representing these
function spaces.

We now state the two main theorems of this paper. In what follows,
$T:BV(\Om)\to\mathcal{F}$, where $\mathcal{F}$ is the collection of all Borel
functions on $\dOm$, is the trace operator as constructed in~\cite{LS},
see~\eqref{defoftrace}. In the event that $\Om$ does not support a
$1$-Poincar\'e inequality, the trace need not be defined for each function in
$BV(\Om)$, but it would still be well-defined in the sense
of~\eqref{defoftrace} for certain functions in $BV(\Om)$. Thus in the next two
theorems, by stating that $T\circ E$ is the identity operator on the
corresponding function space, we are also implicitly claiming that for each $u$
in that function space the trace of $Eu$ is well-defined.
\begin{theorem}\label{thm:main1}
Let $\Om$ be a bounded domain in $X$ satisfying the co-dimension~$1$ Ahlfors
regularity~\eqref{boundary-Ahlfors-regularity}. Then there is a bounded linear
extension operator $E:\Be^0(\dOm)\to BV(\Om)$ such that $T\circ E$ is the
identity operator on $\Be^0(\dOm)$.
\end{theorem}

\begin{theorem}\label{thm:main2}
With $\Om$ a bounded domain in $X$ satisfying the co-dimension~$1$ Ahlfors
regularity~\eqref{boundary-Ahlfors-regularity}, there is a \emph{nonlinear}
bounded extension operator $\Ext:L^1(\dOm)\to BV(\Om)$ such that $T\circ\Ext$
is the identity operator on $L^1(\dOm)$.
\end{theorem}
The extension from $L^1(\dOm)$ to $BV(\Om)$ cannot in general be linear; this
is not an artifact of our proof, see~\cite{Pee, PW},  and the discussion in
Section~4.8 below.

\begin{remark}
In proving Theorem~\ref{thm:main1} we actually prove a stronger but less
elegant statement. We show that if the boundary of $\Om$, equipped with the
co-dimension~$1$ Hausdorff measure $\mathcal{H}$, is \emph{lower} Ahlfors
regular, that is, if $C\mathcal{H}(B(x,r)\cap\partial\Om)\ge \mu(B(x,r))/r$
whenever $x\in\partial\Om$ and $0<r<2\text{diam}(\partial\Om)$, then there is a
bounded linear extension operator $E:\Be^0(\dOm)\to BV(\Om)$. We then show that
in the event that $\partial\Om$ also satisfies the requirement of
\emph{pointwise upper} co-dimension~$1$ Ahlfors regularity in the sense that
for $\mathcal{H}$-almost every $x\in\partial\Om$ there are constants $C(x)\ge
1$ and $R(x)>0$ such that for $0<r\le R(x)$,
\[
\mathcal{H}(B(x,r)\cap\partial\Om)\le C(x)\, \frac{\mu(B(x,r))}{r},
\]
then $T\circ E$ is the identity operator on $\Be^0(\dOm)$.
\end{remark}

Combining the above two theorems with those of~\cite{LS} we obtain the following
identification of the trace class of $BV(\Om)$.
\begin{corollary}
\label{cor:TrBV=L1}
Let $X$ support a $1$-Poincar\'e inequality. With $\Om$ a
bounded domain in $X$ that satisfies the density condition, i.e.,
\begin{equation}\label{density}
 \mu(B(z,r)\cap\Om)\ge C^{-1}\mu(B(z,r))\ \text{ for all } z\in \Om\text{ and } 0<r<\diam(\Om),
\end{equation}
the co-dimension $1$ Ahlfors regularity~\eqref{boundary-Ahlfors-regularity},
and $1$-Poincar\'e inequality, we have that the trace class of $BV(\Om)$ is
$L^1(\dOm)$.
\end{corollary}
In the above we cannot drop any of the respective conditions we impose on the
domain $\Om$. The requirement of the support of a $1$-Poincar\'e inequality is
needed only in order to obtain the trace theorem from~\cite{LS}. As the example
of a slit disc shows, eliminating the support of a Poincar\'e inequality might
result in the failure of the trace theorem, though as the example of the
Euclidean (planar) domain
\[
\Omega=[-2,2]^2\setminus \{(x,y)\in\R^2\, :\, -1\le x\le 1, |x|\le |y|\le 1\}
\]
shows, the support of $1$-Poincar\'e inequality is not essential in obtaining
the trace theorem of~\cite{LS}. The measure density condition~\eqref{density}
is also needed to obtain the trace theorem. Again, this property might not be a
requirement for obtaining the trace theorem of~\cite{LS} and hence the above
corollary, but if the requirement is removed, some other property of the domain
needs to be required as the following example shows. This example is also a
planar Euclidean domain, obtained by pasting a sequence of thin tubes, with
relatively narrower and narrower trunks, to a rectangular base. For each
positive integer $n$ let $U_n$ be the domain given by
\[
 U_n=\left(\frac1{n^2}-4^{-n},\frac1{n^2}+4^{-n}\right)\times[0,2^{-n}),
\]
and let
\[
\Omega=(-1,2)\times (-1,0) \, \cup \bigcup_{n\in\NN} U_n.
\]
The trace theorem of~\cite{LS} fails here because the trace $T(u_n)$ of the
function $u_n=\chi_{(\frac1n-4^{-n},\frac1n+4^{-n})\times[0,2^{-n}]}$ has
$L^1(\partial\Omega)$-norm of the order of $2^{-n}$, while the BV-norm of $u_n$
is of the order of $4^{-n}$. Note that $\Omega$ fails the measure density
condition~\eqref{density}.

For clarity, we note that the statements of Theorems~\ref{thm:main1}
and~\ref{thm:main2} do not require the domains or the ambient metric measure
space to support any Poincar\'e inequality, which allows the domains to have
interior cusps or slits. Thus, even in the Euclidean setting our methods give
rise to new results, as the results of~\cite{Gag} and~\cite{Vit} are in the
setting of Lipschitz domains. Smooth bounded Euclidean domains and bounded
smooth domains in a Riemannian manifold with positive Ricci curvature would
satisfy the hypotheses listed in the above three results. Indeed, such domains
are uniform domains, and as uniform domains in a metric measure space
supporting a $1$-Poincar\'e inequality do support a $1$-Poincar\'e inequality
(see~\cite{BSh}), the trace class of the class of BV functions on such domains
is the $L^1$-class of the boundary of the domain. In general, balls in the
space with center in the smooth domain need not be connected, but there is a
scaling factor $\lambda>0$ such that each ball with center in the domain can be
connected in the $\lambda$-times enlarged ball (that is, all the points in the
original ball belongs to the same connected component of the enlarged ball).
The property of connecting a ball inside a fixed scaled concentric ball is
called \emph{linear local connectivity} in~\cite{HK}. Thus the scaling factor
$\lambda$ on the right-hand side of the Poincar\'e inequality given in
Definition~\ref{def:PI} cannot in general be removed.

A related problem is to investigate the extensions of functions from a domain
$\Omega$ to the whole space. See~\cite{BM, BS, HKT, L, Maz}.
\vskip 1.5ex plus 0.5ex minus 0.5ex
\noindent{\bf Acknowledgement:} The research of the first author was supported
by the Knut and Alice Wallenberg Foundation (Sweden), and the research of the
second author was partially supported by the NSF grant DMS-1500440 (U.S.A.).
Majority of the research for this paper was conducted during the visit of the
third author to the University of Cincinnati; she wishes to thank that
institution for its kind hospitality. The authors also thank the anonymous
referee for suggestions that helped improve the exposition of the paper.
\subsection{Notation and definitions}
In this section $(X,d,\mu)$ denotes a metric measure space with $\mu$ a Radon
measure. We say that $\mu$ is \emph{doubling} if there is a constant $C_D$ such
that for each $z\in Z$ and $r>0$,
\[
   0<\mu(B(z,2r))\le C_D\, \mu(B(z,r))<\infty.
\]

Given a Lipschitz function $f$ on a subset $A\subset X$, we set
\[
\LIP (f,A):=\sup_{x,y\in A\, :\, x\ne y}\frac{|f(x)-f(y)|}{d(x,y)}.
\]
When $x$ is a point in the interior of $A\subset X$, we set
\[
\Lip f(x):=\limsup_{y\to x}\frac{|f(y)-f(x)|}{d(y,x)}.
\]

We follow~\cite{M} to define the function class $BV(X)$. The space $BV(X)$ of
functions of bounded variation consists of functions in $L^1(X)$ that also have
finite total variation on $X$. The total variation of a function on a metric
measure space is measured using upper gradients; the notion of upper gradients,
first formulated in~\cite{HK} (with the terminology ``very weak gradients''),
plays the role of $|\nabla u|$ in the metric setting where no natural
distributional derivative structure exists. A Borel function $g:X\to[0,\infty]$
is an upper gradient of $u:X\to\RR\cup\{\pm\infty\}$ if the following
inequality holds for all (rectifiable) curves $\gamma:[a,b]\to X$, (denoting
$x=\gamma(a)$ and $y=\gamma(b)$),
\[
  |u(y)-u(x)|\leq\int_{\gamma}g\,ds
\]
whenever $u(x)$ and $u(y)$ are both finite, and $\int_{\gamma}g\,ds=\infty$
otherwise. For each function $u$ as above, we set $I(u:X)$ to be the infimum of
the quantity $\int_Xg\, d\mu$ over all upper gradients (in $X$) $g$ of $u$.
\begin{remark}\label{rem:lip-upp1}
We note here that if $u$ is a (locally) Lipschitz function on $Z$, then $\Lip
u$ is an upper gradient of $u$; see for example~\cite{Hei01}. We refer the
interested reader to~\cite{BB, HKST} for more on upper gradients.
\end{remark}
The \emph{total variation} of the function $u\in L^1(X)$ is given by
\[
  \|Du\|(X):=\inf\left\{\liminf_{i\to\infty}I(u_i:X)\colon u_i\in\Lip_{\textrm{loc}}(X),
u_i\to u\textrm{ in }L^1(X)\right\}.
\]
\begin{remark}\label{rem:lip-upp2}
From Remark~\ref{rem:lip-upp1} we know that if $u$ is a locally Lipschitz
continuous function on $X$, then $\Vert Du\Vert(X)\le \int_Z\Lip u\, d\mu$.
\end{remark}
For each open set $U\subset Z$ we can set $\Vert Du\Vert(U)$ similarly:
\[
  \|Du\|(U):=\inf\left\{\liminf_{i\to\infty}I(u_i:U)\colon u_i\in\Lip_{\textrm{loc}}(U), u_i\to u\textrm{ in }L^1(U)\right\}.
\]
It was shown in~\cite{M} that if $\Vert Du\Vert(X)$ is finite, then $U\mapsto
\Vert Du\Vert(U)$ is the restriction of a Radon measure to open sets of $X$. We
use $\Vert Du\Vert$ to also denote this Radon measure.
\begin{definition}
The space $BV(X)$ of functions of bounded variation is equipped with the norm
\[ \|u\|_{BV(X)}:=\|u\|_{L^1(X)}+\|Du\|(X). \]
\end{definition}
This definition of BV agrees with the standard notion of BV functions in the
Euclidean setting, see~\cite{AFP, EvaG92, Zie89}. See also \cite{AMP} for more
on the BV class in the metric setting.

We say that a measurable set $E\subset X$ is of \emph{finite perimeter} if
$\chi_E\in BV(X)$. It follows from~\cite{M} that the superlevel set
$E_t:=\{z\in Z\, :\, u(z)>t\}$ has finite perimeter for almost every $t\in\R$
and that the coarea formula
\[
  \Vert Du\Vert(A)=\int_\R\Vert D\chi_{E_t}\Vert(A) \, dt
\]
holds true whenever $A\subset X$ is a Borel set.
\begin{definition}[cf.\@ \cite{A1}]\label{def:PI}
A metric space $X$ supports a $1$-Poincar\'{e} inequality if there exist
positive constants $\lambda$ and $C$ such that for all balls $B\subset X$ and
all $u\in L^1_\loc(X)$,
\[
\fint_{B}|u-u_{B}|\,d\mu\leq C\rad(B)\frac{\|Du\|(\lambda B)}{\mu(\lambda B)}\,.
\]
\end{definition}
Here and in the rest of the paper, $f_A$ denotes the \emph{integral mean} of a
function $f\in L^0(X)$ over a measurable set $A \subset X$ of finite positive
measure, defined as
\[
  f_A = \fint_A f\,d\mu = \frac{1}{\mu(A)} \int_A f\,d\mu
\]
whenever the integral on the right-hand side exists, not necessarily finite
though. Furthermore, given a ball $B=B(x,r) \subset X$ and $\lambda>0$, the
symbol $\lambda B$ denotes the inflated ball $B(x, \lambda r)$.

Given $A\subset X$, we define its \emph{co-dimension $1$ Hausdorff measure}
$\mathcal{H}(A)$ by
\begin{equation}\label{eq:deff-mathcal}
\mathcal{H}(A)=\lim_{\delta\to 0^+}\ \inf\bigg\lbrace\sum_i\frac{\mu(B_i)}{\rad(B_i)}\, :\, B_i\text{ balls in }X, \rad(B_i)<\delta,
               A\subset \bigcup_iB_i\bigg\rbrace.
\end{equation}

It was shown in~\cite{A1} that if $\mu$ is doubling and supports a
$1$-Poincar\'e inequality, then there is a constant $C\ge 1$ such that whenever
$E\subset X$ is of finite perimeter,
\[
 C^{-1}\mathcal{H}(\partial_m E)\le \Vert D\chi_E\Vert(X)\le C\, \mathcal{H}(\partial_mE),
\]
where $\partial_mE$ is the \emph{measure-theoretic boundary} of $E$. It
consists of those points $z\in X$ for which
\[
\limsup_{r\to 0^+}\frac{\mu(B(z,r)\cap E)}{\mu(B(z,r))}>0\quad
\text{and}\quad
\limsup_{r\to 0^+}\frac{\mu(B(z,r)\setminus E)}{\mu(B(z,r))}>0.
\]

We next turn our attention to the definition of other function spaces to be
considered in this paper. The Besov classes, much studied in the Euclidean
setting, made their first appearance in the metric setting in~\cite{BourPaj}
and were explored further in~\cite{GKS}.
\begin{definition}
Let $(Z,\dd)$ be a metric space equipped with a Radon measure $\nu$. For a
fixed $R>0$, the \emph{Besov space} $B_{1,1}^\theta(Z)$ of smoothness $\theta
\in [0, 1]$ consists of functions of finite Besov norm that is given by
\begin{equation}
  \label{eq:Besov}
  \|u\|_{B_{1,1}^\theta(Z)} = \|u\|_{L^1(Z)} + \int_0^R \int_Z \fint_{B(x, t)} |u(y) - u(x)|\,d\nu(y) \,d\nu(x) \frac{dt}{t^{1+\theta}}\,.
\end{equation}

In our application of Besov spaces, the metric space $Z$ will be the boundary
of a bounded domain in $X$, and the measure $\nu$ will be the restriction of
the co-dimension $1$ Hausdorff measure $\mathcal{H}$ to this boundary.

We will show that the function class $B^\theta_{1,1}(Z)$ is in fact independent
of the choice of $R \in (0, \infty)$, see Lemma~\ref{lem:BesovIndepR} below.

The following \emph{fractional John--Nirenberg space} was first generalized to
the metric measure space setting in~\cite{HKT}. In the Euclidean setting it was
first studied in~\cite{BBM} and~\cite{Brez}, but the case $\theta=0$ in the
Euclidean setting appeared in the earlier work of John and Nirenberg~\cite{JN}.
The fractional John--Nirenberg space $A^{\theta}_{1,\tau}(Z)$, where $\theta
\in[0,1]$ is its smoothness and $\tau \ge 1$ the dilation factor, is defined
via its norm
\begin{equation}
  \label{eq:JN}
  \|u\|_{A^{\theta}_{1,\tau}(Z)} = \|u\|_{L^1(Z)} + \sup_{\bcal_\tau} \sum_{B \in \bcal_\tau}
     \frac{1}{\rad(B)^\theta} \int_{\tau B} |u-u_{\tau B}|\,d\nu,
\end{equation}
where the supremum is taken over all collections $\bcal_\tau$ of balls in $Z$
of radius at most $R/\tau$ such that $\tau B_1\cap\tau B_2$ is empty whenever
$B_1,B_2\in\bcal_\tau$ with $B_1\ne B_2$. The class $A^{\theta}_{1,\tau}(Z)$ is
also independent of the exact choice of $R \in (0, \infty)$.
\end{definition}
\subsection{Standing assumptions}
Throughout this paper $(X,d,\mu)$ is a metric measure space, with $\mu$ a Borel
regular measure. We assume that $X$ is complete and that $\mu$ is doubling on
$X$. Furthermore, $\Omega\subset X$ is a bounded domain and there is a constant
$C\ge 1$ such that for all $x\in \dOm$ and $0<r\le \diam(\Om)$, we have
\begin{equation}\label{boundary-Ahlfors-regularity}
 C^{-1}\frac{\mu(B(x,r))}{r}\le \mathcal{H}(B(x,r)\cap\dOm)\le C\frac{\mu(B(x,r))}{r}.
\end{equation}
The property of satisfying~\eqref{boundary-Ahlfors-regularity} will be called
\emph{Ahlfors codimension $1$ regularity} of $\dOm$.

Throughout the paper $C$ represents various constants that depend solely on the
doubling constant, constants related to the Poincar\'e inequality, and the
constants related to~\eqref{boundary-Ahlfors-regularity}. The precise value of
$C$ is not of interest to us at this time, and its value may differ in each
occurrence. Given expressions $a$ and $b$, we say that $a\approx b$ if there is
a constant $C\ge 1$ such that $C^{-1}a\le b\le C a$.
\section{Bounded linear extension from Besov class to BV class: proof of Theorem~\ref{thm:main1}}
\label{sec:B110-extension}
\subsection{Whitney cover and partition of unity}
The following theorem from~\cite[Section 4.1]{HKST} gives the existence of a
Whitney \textit{covering} of an open subset $\Omega$ of a doubling metric space
$X$ by balls whose radii are comparable to their distance from the boundary,
see also \cite{BBS}.
\begin{theorem}
Let $\Omega\subsetneq X$ be bounded and open.  Then there exists a countable
collection $\mathcal{W}_\Omega=\{B(p_{j, i},r_{j,i})=B_{j,i}\}$ of balls in
$\Omega$ so that
\begin{itemize}
  \setlength{\parskip}{0pt}
  \setlength{\itemsep}{2pt plus 1pt minus 2pt}
 \item $\bigcup_{j,i} B_{j,i} = \Omega$,
 \item $\sum_{j,i} \chi_{B(p_{j,i}, 2r_{j,i})}\leq 2C_D^5$,
 \item $2^{j-1}<r_{j,i}\leq 2^{j}$ for all $i$,
 \item and so that $r_{j,i}=\frac{1}{8}\dist(p_{j,i}, X\setminus \Omega)$.
\end{itemize}
Here the constant $C_D$ is the doubling constant of the measure $\mu$.
\end{theorem}
The radii of the balls are small enough so that $2B_i\subset\Omega$. Also,
since we are only concerned with bounded domains $\Omega$, there is a largest
exponent $j$ that occurs in the cover; we denote this exponent by $j_0$. Hence
$-j\in\NN\cup\{0,\cdots, -j_0\}$. Note that $2^{j_0}$ is comparable to
$\diam(\Om)$. One wishing to keep track of the relationships between various
constants should therefore keep in mind that the constants that depend on $j_0$
then depend on $\diam(\Om)$.

We also note that no ball in level $j$ intersects a ball in level $j+2$.  This
follows by the reverse triangle inequality $d(p_{j,i},p_{j+2,k})\geq
2^{j+4}-2^{j+3}=2^{j+3}$ and the bounds on the radii: $2^{j-1}<r_{j,i}\leq 2^j$
and $2^{j+1}<r_{j+2,k}\leq 2^{j+2}$.

As in~\cite[Section~4.1]{HKST}, there is a Lipschitz partition of unity
$\{\phi_{j,i}\}$ subordinate to the Whitney decomposition $\wcal_\Om$, that is,
$\sum_{j,i} \phi_{j,i} \equiv \chi_\Om$ and for every ball
$B_{j,i}\in\wcal_\Om$, we have that $\chi_{B_{j,i}} \le \phi_{j,i} \le
\chi_{2B_{j,i}}$ and $\phi_{j,i}$ is $C/r_{j,i}$-Lipschitz continuous.
\subsection{An extension of Besov functions}
\label{ssec:BesovExt}
Suppose that $f:\partial\Omega\to\RR$ is a function in
$B_{1,1}^0(\partial\Omega)$. We want to define a function $F:\Omega\to\RR$
whose trace is the original function $f$ on $\partial\Omega$.

Consider the center of the Whitney ball $p_{j,i}\in\Omega$ and choose a closest
point $q_{j,i}\in\partial\Omega$. Define $U_{j,i}:=B(q_{j,i}, r_{j,i})\cap
\partial\Omega$. We set $a_{j,i}:=\fint_{U_{j,i}}f(y)\,d\mathcal{H}(y)$. Then
for $x\in\Omega$ set
\[
F(x):=\sum_{j,i} a_{j,i} \phi_{j,i}.
\]
In subsequent results in this section we will show that $F\in BV(\Om)$. From
the following proposition and Remark~\ref{rem:lip-upp2} we obtain the desired
bound for $\Vert DF\Vert(\Om)$.
\begin{proposition}
\label{prop:extnBounds}
Given $\Omega\subset X$ and $f\in B_{1,1}^0(\partial\Omega)$, there exists
$C>0$ such that
\[
  \int_\Omega \Lip F\,d\mu \leq C\|f\|_{B_{1,1}^0(\partial\Omega)}.
\]
\end{proposition}
\begin{proof}
Fix a ball $B_{\ell,m}\in \wcal_\Om$, and fix a point $x \in B_{\ell,m}$.  For
all $y\in B_{\ell,m}$,
  \begin{align*}
    |F(y)-F(x)| = \left|\sum_{j,i}a_{j,i}(\phi_{j,i}(y)-\phi_{j,i}(x))\right|& \\
     = \left|\sum_{j,i}(a_{j,i}-a_{\ell,m})(\phi_{j,i}(y)-\phi_{j,i}(x))\right|
    \leq &\sum_{\substack{j,i\, \textrm{ s.t.}\\ 2B_{j,i}\cap B_{\ell,m}\ne\emptyset}}|a_{j,i}-a_{\ell,m}|\frac{C}{r_{j,i}}d(y,x).
  \end{align*}
The last inequality in the above sequence follows from the Lipschitz constant
of $\phi_{j,i}$.  Rearranging and noting that if the balls intersect then
$|j-\ell|\le 1$, we see that
  \[
  \frac{|F(y)-F(x)|}{d(y,x)}\leq \frac{C}{r_{\ell,m}} \sum_{\substack{j,i \,\textrm{ s.t.} \\
    2B_{j,i}\cap B_{\ell,m}\neq \emptyset}}|a_{j,i}-a_{\ell,m}|.
  \]
Hence, we want to bound terms of the form $|a_{j,i}-a_{\ell,m}|$:
\begin{align}
  |a_{j,i}-a_{\ell,m}| &= \left|\fint_{U_{j,i}}f(z)\,d\mathcal{H}(z)-\fint_{U_{\ell,m}}f(z)\,d\mathcal{H}(z)\right|\notag \\
  &=   \left|\fint_{U_{j,i}}\fint_{U_{\ell,m}}\bigl(f(z)-f(w)\bigr)\,d\mathcal{H}(w)\,d\mathcal{H}(z)\right|  \notag\\
  &\le  \fint_{U_{j,i}}\fint_{U_{\ell,m}}  \left|f(z)-f(w)\right|\,d\mathcal{H}(w)\,d\mathcal{H}(z) \notag \\
  &= \frac{1}{\mathcal{H}(U_{j,i})\mathcal{H}(U_{\ell,m})}\int_{U_{j,i}}\int_{U_{\ell,m}}
      \left|f(z)-f(w)\right|\,d\mathcal{H}(w)\,d\mathcal{H}(z) \notag \\
  \label{ineq:expandballs1}
  &\leq \frac{C}{\mathcal{H}(U_{\ell,m}^*)\mathcal{H}(U_{\ell,m}^*)}\int_{U_{j,i}}\int_{U_{\ell,m}}
     \left|f(z)-f(w)\right|\,d\mathcal{H}(w)\,d\mathcal{H}(z)  \\
  &\leq \frac{C}{\mathcal{H}(U_{\ell,m}^*)\mathcal{H}(U_{\ell,m}^*)}\int_{U^*_{\ell,m}}\int_{U^*_{\ell,m}}
      \left|f(z)-f(w)\right|\,d\mathcal{H}(w)\,d\mathcal{H}(z) \notag\\
  &= C\fint_{U^*_{\ell,m}}\fint_{U^*_{\ell,m}}  \left|f(z)-f(w)\right|\,d\mathcal{H}(w)\,d\mathcal{H}(z),\notag
\end{align}
where $U^*_{\ell,m}$ denotes the expanded subset of the boundary:
\begin{equation}
\label{eq:expandedsubset}
  U^*_{\ell,m}:=B(q_{\ell,m}, 2^{6}r_{\ell,m})\cap\partial\Omega.
\end{equation}
By the doubling property of $X$, the boundary regularity condition on
$\partial\Omega$, and the definition of codimension-$1$ Hausdorff measure, we
have
\[
\mathcal{H}(U^*_{\ell,m})\leq C\mathcal{H}(U_{\ell,m})\,,
\]
which gave inequality~\eqref{ineq:expandballs1}. The above estimates together
with the bounded overlap of the Whitney balls yield the following inequality:
\begin{align}\label{eq:pointwise-Lip-est}
\Lip F(x)&=\limsup_{y\to x}\frac{|F(y)-F(x)|}{d(y,x)}\notag\\
   &\le \frac{C}{r_{\ell,m}}\,  \fint_{U^*_{\ell,m}}\fint_{U^*_{\ell,m}}  \left|f(z)-f(w)\right|\,d\mathcal{H}(w)\,d\mathcal{H}(z)
\end{align}
for $x\in B_{\ell,m}$. From~\eqref{eq:pointwise-Lip-est}
and~\eqref{boundary-Ahlfors-regularity} we see that
\begin{align*}
 \int_\Omega \Lip F(x)\,d\mu(x) &\le \sum_{\ell,m}\int_{B_{\ell,m}}\Lip F(x)\, d\mu(x)\notag\\
      &\le \sum_{\ell, m}\mu(B_{\ell,m})\, \frac{C}{r_{\ell,m}}\,
          \fint_{U^*_{\ell,m}}\fint_{U^*_{\ell,m}}  \left|f(z)-f(w)\right|\,d\mathcal{H}(w)\,d\mathcal{H}(z)\notag\\
 &\leq C\sum_{\ell,m} \hcal(U_{\ell,m})
 \fint_{U^*_{\ell,m}}\fint_{U^*_{\ell,m}}  \left|f(z)-f(w)\right|\,d\mathcal{H}(w)\,d\mathcal{H}(z)\\  &\leq C\sum_{\ell=-\infty}^{j_0}\sum_{m}
 \int_{U^*_{\ell,m}}\fint_{U^*_{\ell,m}}  \left|f(z)-f(w)\right|\,d\mathcal{H}(w)\,d\mathcal{H}(z)\\
 &\leq {C}\sum_{\ell=-\infty}^{j_0}
 \int_{\partial\Omega}\fint_{B(z,2^{7+\ell})}  \left|f(z)-f(w)\right|\,d\mathcal{H}(w)\,d\mathcal{H}(z)\,.
\end{align*}
Here the last inequality follows from the uniformly bounded overlap of the
balls $U_{\ell,m}^*$ for each $\ell$. Without loss of generality, we may choose
$R = 2^{j_0+7}$ in the definition of the Besov norm \eqref{eq:Besov}. Note that
$R\approx \diam(\Omega)$ then. The following estimate (cf.\@ the proof
of~\cite[Theorem~5.2]{GKS}) concludes the proof:
\begin{align}
\sum_{\ell=-\infty}^{j_0}&
 \int_{\partial\Omega}\fint_{B(z,2^{7+\ell})}  \left|f(z)-f(w)\right|\,d\mathcal{H}(w)\,d\mathcal{H}(z)\notag\\
 & \label{eq:BesovEquivSum}
   \approx\int_{t=0}^{2^{j_0+7}}
   \int_{\partial\Omega}\fint_{B(z,t)}  \left|f(z)-f(w)\right|\,d\mathcal{H}(w)\,d\mathcal{H}(z)\frac{dt}{t} \\
 & \le C\|f\|_{B_{1,1}^0(\partial\Omega)}\,. \notag \qedhere
\end{align}
\end{proof}
We will use the extension constructed in this section in formulating a
nonlinear bounded extension from $L^1(\dOm,\mathcal{H})$ to $BV(\Om)$ in the
subsequent sections. There we will need the following estimates for the
integral of the gradient and the function on layers of $\Omega$.
\begin{lemma}\label{layer-est-grad}
For $0 \le \rho_1<\rho_2<\diam(\Om)/2$, set
\begin{equation}
  \label{eq:def-OmRhoRho}
  \Om(\rho_1,\rho_2):=\{x\in\Om\, :\, \rho_1\le \dist(x,X\setminus\Om)<\rho_2\}.
\end{equation}
Let $\mathcal{J}(\rho_1,\rho_2)$ be the collection of all $\ell\in\ZZ$ such
that there is some $m\in\NN$ with $B_{\ell,m}\cap\Omega(\rho_1,\rho_2)$
non-empty. Then
\[
\int_{\Om(\rho_1,\rho_2)}\Lip F\, d\mu\le C\, \sum_{\ell\in\mathcal{J}(\rho_1,\rho_2)}
    \int_{\dOm}\fint_{B(z,2^{7+\ell})}|f(z)-f(w)|\, d\mathcal{H}(w)\,d\mathcal{H}(z).
\]
\end{lemma}
\begin{proof}
For each $\ell\in\mathcal{J}(\rho_1,\rho_2)$ let $\mathcal{I}(\ell)$ denote the
collection of all $m\in\NN$ for which $B_{\ell,m}\cap\Omega(\rho_1,\rho_2)$ is
non-empty. Then by~\eqref{eq:pointwise-Lip-est}
and~\eqref{boundary-Ahlfors-regularity},
\begin{align*}
\int_{\Om(\rho_1,\rho_2)}\Lip F\, d\mu
  &\le \sum_{\ell\in\mathcal{J}(\rho_1,\rho_2)}\sum_{m\in\mathcal{I}(\ell)}\int_{B_{\ell,m}}\Lip F\, d\mu\\
  \le C \sum_{\ell\in\mathcal{J}(\rho_1,\rho_2)}&\sum_{m\in\mathcal{I}(\ell)}\frac{\mu(B_{\ell,m})}{r_{\ell,m}}\,
          \fint_{U^*_{\ell,m}}\fint_{U^*_{\ell,m}}  \left|f(z)-f(w)\right|\,d\mathcal{H}(w)\,d\mathcal{H}(z)\\
   \le C \sum_{\ell\in\mathcal{J}(\rho_1,\rho_2)}&\sum_{m\in\mathcal{I}(\ell)} \hcal(U_{\ell,m})
        \fint_{U^*_{\ell,m}}\fint_{U^*_{\ell,m}}  \left|f(z)-f(w)\right|\,d\mathcal{H}(w)\,d\mathcal{H}(z)\\
   = C \sum_{\ell\in\mathcal{J}(\rho_1,\rho_2)}&\sum_{m\in\mathcal{I}(\ell)}
        \int_{U^*_{\ell,m}}\fint_{U^*_{\ell,m}}  \left|f(z)-f(w)\right|\,d\mathcal{H}(w)\,d\mathcal{H}(z)\\
    \le C \sum_{\ell\in\mathcal{J}(\rho_1,\rho_2)}&\sum_{m\in\mathcal{I}(\ell)}
        \int_{U^*_{\ell,m}}\fint_{B(z,2^{7+\ell})}|f(z)-f(w)|\, d\mathcal{H}(w)\,d\mathcal{H}(z)\\
    \le C \sum_{\ell\in\mathcal{J}(\rho_1,\rho_2)}&\int_{\dOm}\fint_{B(z,2^{7+\ell})}|f(z)-f(w)|\, d\mathcal{H}(w)\,d\mathcal{H}(z).
\qedhere
\end{align*}
\end{proof}
\begin{corollary}
\label{cor:layer-est-grad-lip} %
Using the notation of Lemma~\ref{layer-est-grad}, we have that
\[
  \int_{\Om(\rho_1,\rho_2)}\Lip F\, d\mu\le C \rho_2 \hcal(\dOm) \LIP(f,\dOm)
\]
whenever $f$ is Lipschitz on $\dOm$.
\end{corollary}
\begin{proof}
For a fixed $\ell \in \ZZ$, we can estimate
\begin{align*}
  \int_{\dOm}&\fint_{B(z,2^{7+\ell})}|f(z)-f(w)|\, d\hcal(w)\,d\hcal(z) \\
  & \le \int_{\dOm}\fint_{B(z,2^{7+\ell})}\LIP(f,\dOm) \dd(z,w) \, d\hcal(w)\,d\hcal(z) \\
  & \le C \hcal(\dOm) \, \LIP(f,\dOm)\,  2^{7+\ell}\,.
\end{align*}
Therefore,
\[
\int_{\Om(\rho_1,\rho_2)}\Lip F\, d\mu\, \le\, C \hcal(\dOm) \LIP(f,\dOm) \, \sum_{\ell\in\mathcal{J}(\rho_1,\rho_2)} 2^\ell.
\]
Every ball $B=B(p,r) \in \mathcal{I}(\ell)$ satisfies $2^{\ell-1} < r \le
2^\ell$ and  $\dist(p, X\setminus \Om) = 8 r$. There is $C\ge1$ such that
$C^{-1} \rho_1 \le 2^\ell \le C \rho_2$ whenever $\ell \in \mathcal{J}(\rho_1,
\rho_2)$. Thus, $\sum_{\ell\in\mathcal{J}(\rho_1,\rho_2)} 2^\ell \le C \rho_2$.
\end{proof}
We next turn our attention to the $L^1$-estimates for $F$.
\begin{lemma}
\label{lem:L1-est_Whitney}
There exists $C>0$ such that
\[\int_\Omega |F|\,d\mu \leq C\diam(\Om) \|f\|_{L^1(\dOm)}.\]
\end{lemma}
\begin{proof}
We first consider a fixed ball $B_{\ell, m}$ from the Whitney cover. Then
\begin{align*}
\int_{B_{\ell, m}}|F(x)|\,d\mu(x) &=\int_{B_{\ell, m}}\left|\sum_{j,i} \fint_{U_{j,i}}f(y)\,d\hcal(y)\phi_{j,i}(x)\right|\,d\mu(x)\\
&\leq\int_{B_{\ell, m}}\sum_{j,i}\left|\fint_{U_{j,i}}f(y)\,d\hcal(y)\right|\phi_{j,i}(x)\,d\mu(x)\\
&=\int_{B_{\ell, m}}  \sum_{\substack{j,i \, \textrm{ s.t.} \notag\\
    2B_{j,i}\cap B_{\ell,m}\neq \emptyset}}\left|\fint_{U_{j,i}}f(y)\,d\hcal(y)\right|\phi_{j,i}(x)\,d\mu(x).
\end{align*}

Recall that if $2B_{j,i}\cap B_{\ell,m}\neq \emptyset$, then $|j-\ell|\le 1$,
so $\mathcal{H}(U_{j,i})\approx \mathcal{H}(U_{\ell,m}^*)$. Also, for $U_{\ell,
m}^*$ as defined in equation~\eqref{eq:expandedsubset}, $U_{j,i}\subset
U_{\ell,m}^*$. Furthermore, by the construction of the Whitney decomposition,
each point is in a fixed number of dilated Whitney balls $2B_{j,i}$. Hence,
\begin{align*}
 \int_{B_{\ell, m}}  \sum_{\substack{j,i \,\textrm{ s.t.} \\
    2B_{j,i}\cap B_{\ell,m}\neq \emptyset}}\biggl|\fint_{U_{j,i}} f(y)\,d\hcal(y)\biggr| \phi_{j,i}(x)\,d\mu(x)&\\
\leq C \int_{B_{\ell, m}} \fint_{U_{\ell,m}^*}|f(y)|\,d\hcal(y)\, d\mu(x)
\leq& C \mu(B_{\ell, m})  \fint_{U_{\ell,m}^*}|f(y)|\,d\hcal(y).
\end{align*}
In view of~\eqref{boundary-Ahlfors-regularity}, we obtain that
\begin{equation}\label{eq:Ball-F-est}
\int_{B_{\ell, m}}|F(x)|\,d\mu(x)\le C\,r_{\ell,m} \int_{U_{\ell,m}^*}|f(y)|\,d\hcal(y).
\end{equation}
Summing up and noting that $\Omega=\bigcup_{\ell,m}B_{\ell,m}$, we have
\begin{align*}
  \int_\Omega |F|\,d\mu
      \leq C \sum_{\ell=-\infty}^{j_0}\sum_{m} r_{\ell,m} \int_{U_{\ell,m}^*}\left|f\right|\,d\hcal
      \leq C \sum_{\ell=-\infty}^{j_0}&2^{\ell} \sum_{m} \int_{U_{\ell,m}^*}\left|f\right|\,d\hcal \\
      \leq C \sum_{\ell=-\infty}^{j_0}2^{\ell}  \int_{\partial\Omega}\left|f\right|\,d\hcal
      \le& C \diam(\Om) \|f\|_{L^1(\partial\Omega)}.
  \pushQED{}\tag*{$\square$}
\end{align*}
\end{proof}
We now aim to obtain an analog of Lemma~\ref{layer-est-grad} for the $L^1$-norm
of $F$ on the layer $\Om(\rho_1,\rho_2)$.
\begin{lemma}\label{layer-est-Fn-boundaryball}
Let $z \in \dOm$ and $r \in (0, \diam(\Om)/2)$. Then,
\[
\int_{B(z,r)\cap \Om(\rho_1,\rho_2)}|F|\, d\mu\le C \, \min\{r, \rho_2\}\, \int_{B(z, 2^8r) \cap \dOm}|f|\, d\mathcal{H}
\]
whenever $0\le \rho_1 < \min\{r, \rho_2\}$ and $\rho_2 < \diam(\Om)/2$.
\end{lemma}
\begin{proof}
Since $B(z,r)\cap \Om(\rho_1,\rho_2) = B(z,r)\cap
\Om(\rho_1,\min\{r,\rho_2\})$, we do not lose any generality by assuming that
$\rho_2 \le r$. Similarly as in the proof of Lemma~\ref{layer-est-grad}, we set
$\mathcal{J}^\prime(\rho_1,\rho_2)$ to be the collection of all $\ell$ for
which there is some $m$ such that $B_{\ell,m}\cap\Omega(\rho_1,\rho_2)\cap
B(z,r)$ is non-empty, and for each $\ell\in\mathcal{J}^\prime(\rho_1,\rho_2)$
we set $\mathcal{I}'(\ell)$ to be the collection of all $m\in\NN$ for which
$B_{\ell,m}\cap\Omega(\rho_1,\rho_2) \cap B(z,r)$ is non-empty. Then
by~\eqref{eq:Ball-F-est},
\begin{align*}
\int_{B(z,r) \cap\Om(\rho_1,\rho_2) } |F|\, d\mu
  & \le \sum_{\ell\in \mathcal{J}^\prime(\rho_1,\rho_2)}\sum_{m\in\mathcal{I}'(\ell)}\int_{B_{\ell,m}}|F|\, d\mu\\
  & \le C \sum_{\ell\in \mathcal{J}^\prime(\rho_1,\rho_2)}\sum_{m\in\mathcal{I}'(\ell)} r_{\ell,m}\int_{U^*_{\ell,m}}|f|\, d\mathcal{H}.
\end{align*}
The triangle inequality yields that
\[
  d(z, q_{\ell, m}) \le d(z, p_{\ell,m}) + d(p_{\ell,m}, q_{\ell, m}) \le 2 d(z, p_{\ell,m}) \le 2(r + r_{\ell,m}),
\]
where $B_{\ell, m} = B(p_{\ell,m}, r_{\ell,m})$ and $U_{\ell, m} =
B(q_{\ell,m}, r_{\ell,m}) \cap \dOm$ with $q_{\ell,m} \in \dOm$ being a
boundary point lying closest to $p_{\ell,m}$. Moreover, $8 r_{\ell,m} =
\dist(p_{\ell,m}, X\setminus \Om) \le d(p_{\ell,m}, z) \le r+ r_{\ell,m}$.
Hence, $r_{\ell,m} \le \tfrac17 r$. Consequently, $d(z, q_{\ell,m}) \le
\tfrac{16}{7} r$ and $U_{\ell, m} \subset B(z, (\tfrac{16}{7} + \tfrac17)r)$.
Thus, $U^*_{\ell, m} \subset B(z, 2^8 r)$ and
\begin{align*}
  \int_{B(z,r) \cap\Om(\rho_1,\rho_2) } |F|\, d\mu
  & \le C\sum_{\ell\in \mathcal{J}^\prime(\rho_1,\rho_2)}2^{\ell}\, \int_{B(z, 2^8 r) \cap \dOm}|f|\, d\mathcal{H} \\
  & \le C \rho_2 \int_{B(z, 2^8 r) \cap \dOm}|f|\, d\mathcal{H},
\end{align*}
where the last inequality can be verified as follows: Every ball $B=B(p,r) \in
\mathcal{I}^\prime(\ell)$ satisfies $2^{\ell-1} < r \le 2^\ell$ and  $\dist(p,
X\setminus \Om) = 8 r$. There is $C\ge1$ such that $C^{-1} \rho_1 \le 2^\ell
\le C \rho_2$ whenever $\ell\in\mathcal{J}^\prime(\rho_1,\rho_2)$. Thus,
$\sum_{\ell\in\mathcal{J}^\prime(\rho_1,\rho_2)} 2^\ell \le C \rho_2$.
\end{proof}

By covering $\dOm$ by balls of radii $r$, whose overlap is bounded, we obtain
the following corollary.

\begin{corollary}\label{layer-est-Fn}
With the notation of Lemma~\ref{layer-est-grad}, we have
\[
\int_{\Om(\rho_1,\rho_2)}|F|\, d\mu\le C \, \rho_2\, \int_{\dOm}|f|\, d\mathcal{H}.
\]
\end{corollary}
\subsection{Trace of extension is the identity mapping}
From the above lemma we know that given a function $f\in \Be^0(\dOm)$ the
corresponding function $F$ is in the class $N^{1,1}(\Om)\subset BV(\Om)$, where
$N^{1,1}(\Om)$ is a Newtonian class introduced in~\cite{S}. The mapping
$f\mapsto F$ is denoted by the operator $E:\Be^0(\dOm)\to BV(\Om)$. This
operator is bounded by Proposition~\ref{prop:extnBounds} and
Lemma~\ref{lem:L1-est_Whitney}, and it is linear by construction.

We now wish to show that the trace of $F$ returns the original function $f$,
i.e., $T\circ E$ is the identity function on $\Be^0(\dOm)$. It was shown
in~\cite{LS} that if $\Om$ satisfies our standing assumptions, then for each
$u\in BV(\Om)$ and for $\mathcal{H}$-a.e.\@ $z\in\dOm$ there is a number
$Tu(z)\in\R$ such that
\begin{equation}\label{defoftrace}
  \limsup_{r\to0^+}\fint_{B(z,r)\cap\Om}|u(y)-Tu(z)|\, d\mu(y)=0.
\end{equation}
The map $u\mapsto Tu$ is called the \emph{trace} of $BV(\Om)$. Moreover, if
$\Om$ supports a $1$-Poincar\'e inequality, then $Tu\in L^1(\dOm)$ for $u\in
BV(\Om)$.

Note also that $\Be^0(\dOm)\subset L^1(\dOm)$ and the inclusion is strict in
general, which is shown in Example~\ref{exa:Bes0!=L1} below. Further properties
of the Besov classes are explored in Section~\ref{sec:spaces}.

For the sake of clarity, let us explicitly point out that the following lemma
shows that the BV extension of a function of the Besov class $B^0_{1,1}(\dOm)$,
as constructed above, has a well-defined trace even though no Poincar\'e
inequality for $\Om$ or for $X$ is assumed.
\begin{lemma}\label{trace-Besov}
For $f$ and $F$ as above, and for $\hcal$-a.e.\@ $z\in \partial\Omega$,
\[\lim_{r\to 0^+} \fint_{B(z,r)\cap\Omega}|F(x)-f(z)|\,d\mu(x)=0.\]
That is, $TEf(z)$ exists for $\hcal$-a.e.~$z\in\dOm$.
\end{lemma}
\begin{proof}
Since $f\in \Be^0(\dOm)\subset L^1(\dOm)$, we know by the doubling property of
$\mathcal{H}\vert_{\dOm}$ that $\mathcal{H}$-a.e.\@ $z\in\dOm$ is a Lebesgue
point of $f$. Let $z$ be such a point. Since
$\smash{\sum_{j,i}}\phi_{j,i}=\chi_\Om$, we have
\begin{align*}
  \fint_{B(z,r)\cap\Omega}|F-f(z)|\,d\mu & =
  \fint_{B(z,r)\cap\Omega}\left|\sum_{j,i} \left(\fint_{U_{j,i}}f\,d\hcal\right)\phi_{j,i}(x)-f(z)\right|\,d\mu(x) \\
  &=  \fint_{B(z,r)\cap\Omega}\left|\sum_{j,i} \left(\fint_{U_{j,i}}(f-f(z))\,d\hcal\right)\phi_{j,i}(x)\right|\,d\mu(x) \\
  &\le   \fint_{B(z,r)\cap\Omega}\sum_{j,i} \left(\fint_{U_{j,i}}\left|f-f(z)\right|\,d\hcal\right)\phi_{j,i}(x)\,d\mu(x)\,.
\end{align*}
By the properties of the Whitney covering $\wcal_\Om$,
\begin{align*}
\int_{B(z,r)\cap\Om}&\sum_{j,i} \left(\fint_{U_{j,i}}\left|f-f(z)\right|\,d\hcal\right)\phi_{j,i}(x)\,d\mu(x) \\
  & \le \sum_{\substack{\ell,m\, \textrm{ s.t.}\\ B_{\ell,m}\cap B(z,r)\ne\emptyset}} \int_{B_{\ell,m}}
     \sum_{j,i} \left(\fint_{U_{j,i}}\left|f-f(z)\right|\,d\hcal\right)\phi_{j,i}(x)\,d\mu(x)\\
  & \le \sum_{\substack{\ell,m\, \textrm{ s.t.}\\B_{\ell,m}\cap B(z,r)\ne\emptyset}} \int_{B_{\ell,m}}
     \sum_{\substack{j,i\,\textrm{ s.t.} \\ 2B_{j,i}\cap B_{\ell,m}\ne\emptyset}} \left(\fint_{U_{j,i}}\left|f-f(z)\right|\,d\hcal\right)\,d\mu(x).
\end{align*}
If $2B_{j,i}\cap B_{\ell,m}$ is non-empty, then $U_{j,i}\subset U_{\ell,m}^*$
and $\mathcal{H}(U_{j,i})\approx \mathcal{H}(U_{\ell,m}^*)$. Therefore
by~\eqref{boundary-Ahlfors-regularity} we have
\begin{align*}
\int_{B(z,r)\cap\Om}\sum_{j,i} &\left(\fint_{U_{j,i}}\left|f-f(z)\right|\,d\hcal\right)\phi_{j,i}(x)\,d\mu(x) \\
  &\le  C\sum_{\substack{\ell,m\, \textrm{ s.t.}\\ B_{\ell,m}\cap B(z,r)\ne\emptyset}}
       \left(\fint_{U_{\ell,m}^*}\left|f-f(z)\right|\,d\hcal\right)\mu(B_{\ell,m})\\
  &\le C\sum_{\substack{\ell,m\, \textrm{ s.t.}\\ B_{\ell,m}\cap B(z,r)\ne\emptyset}} r_{\ell,m}\int_{U_{\ell,m}^*}\left|f-f(z)\right|\,d\hcal.
\end{align*}

Let $\mathcal{J}(B(z,r))$ denote the collection of all $\ell\in\ZZ$ for which
there is some $m\in\NN$ such that $B_{\ell,m}\cap B(z,r)$ is non-empty. For
each $\ell\in\mathcal{J}(B(z,r))$, set $\mathcal{I}(\ell)$ to be the collection
of all $m\in\NN$ for which $B_{\ell,m}\cap B(z,r)$ is non-empty. Then,
\begin{align*}
\int_{B(z,r)\cap\Om}\sum_{j,i} &\left(\fint_{U_{j,i}}\left|f-f(z)\right|\,d\hcal\right)\phi_{j,i}(x)\,d\mu(x) \\
&\le C\sum_{\ell\in\mathcal{J}(B(z,r))}2^{\ell} \sum_{m\in\mathcal{I}(\ell)}\int_{U_{\ell,m}^*}\left|f-f(z)\right|\,d\hcal\\
&\le C\sum_{\ell\in\mathcal{J}(B(z,r))}2^{\ell} \int_{B(z,2^7r)\cap\dOm}|f-f(z)|\, d\hcal\\
&\le C\, r\, \int_{B(z,2^7r)\cap\dOm}|f-f(z)|\, d\hcal.
\end{align*}
In the above, we used the fact that
$\sum_{\ell\in\mathcal{J}(B(z,r))}2^{\ell}\approx r$, since only the indices
$\ell\in\ZZ$ for which $2^{\ell}\approx\dist(B_{\ell,m},X\setminus\Om)\le r$
are allowed to be in $\mathcal{J}(B(z,r))$. From the fact that $z$ is a
Lebesgue point of $f$, we now have
\begin{align*}
 \fint_{B(z,r)\cap\Omega}|F-f(z)|\,d\mu&\le C \frac{r}{\mu(B(z,r)\cap\Om)}\int_{B(z,2^7r)\cap\dOm}|f-f(z)|\, d\hcal\\
   &\le C\fint_{B(z,2^7r)\cap\dOm}|f-f(z)|\, d\hcal\to 0\text{ as }r\to 0^+.
\end{align*}
This completes the proof of the lemma.
\end{proof}
\section{Comparison of $\Be^0(\dOm)$ and other function spaces}\label{sec:spaces}
We now wish to show that $\Be^0(\dOm)$ has more interesting functions than mere
constant functions. What functions are in $B_{1,1}^0(\partial\Omega)$?  Since
the results of this section deal with function spaces based on more general
doubling metric measure spaces, we consider the underlying metric measure space
$(Z,\dd,\nu)$. The other function spaces include $L^1$, BV, and the fractional
John--Nirenberg spaces as well as the class of Lipschitz functions.

Let $Z = (Z, \dd, \nu)$ be a metric space endowed with a doubling measure. In
applications in this paper, $Z$ will be $\partial \Omega \subset X$ and $\nu$
will be the Hausdorff co-dimension $1$ measure $\mathcal{H}\vert_{\dOm}$.
\subsection{Preliminary results}
\begin{lemma}\label{cor:ballcovering}
For every $\tau > 3$, there is $C = C(C_D, \tau) \ge 1$ such that for every
$r>0$ there is an at most countable set of points $\{x_j\}_j \subset Z$
(alternatively, $\{x_j\}_j \subset \Omega$, where $\Omega \subset Z$ is
arbitrary) such that
\begin{itemize}
  \setlength{\parskip}{0pt}
  \setlength{\itemsep}{2pt plus 1pt minus 2pt}
	\item $B(x_j, r) \cap B(x_k, r) = \emptyset$ whenever $j\neq k$;
  \item $Z = \bigcup_j B(x_j, \tau r)$ (alternatively, $\Omega \subset \bigcup_j B(x_j, \tau r)$);
  \item $\sum_j \chi_{B(x_j, \tau r)} \le C$.
\end{itemize}
\end{lemma}
The above lemma is widely known to experts in the field, but we were unable to
find it in current literature; hence we provide a sketch of its proof.
\begin{proof}
An application of Zorn's lemma or~\cite[Lemma~4.1.12]{HKST} gives a countable
set $A\subset Z$ such that for distinct points $x,y\in A$ we have $\dd(x,y)\ge
r$, and for each $z\in Z$ there is some $x\in A$ such that $\dd(z,x)<r$. The
countable collection $\{B(x,r)\, :\, x\in A\}$ can be seen to satisfy the
requirements set forth in the lemma because of the doubling property of $\nu$.
\end{proof}
\begin{lemma}
\label{lem:BesovIndepR}
Let $f\in L^1(Z)$. Then,
\begin{align*}
  &\int_0^r \int_Z \fint_{B(x, t)} |f(y) - f(x)|\,d\nu(y) \,d\nu(x) \frac{dt}{t^{1+\theta}}
     < \infty \quad\mbox{if and only if} \\
  &\int_0^R \int_Z \fint_{B(x, t)} |f(y) - f(x)|\,d\nu(y) \,d\nu(x) \frac{dt}{t^{1+\theta}} < \infty\,,
\end{align*}
where $0<r<R<\infty$. If $\theta > 0$, then the equivalence holds true even for
$R=\infty$.
\end{lemma}
\begin{proof}
By the triangle inequality, we obtain for $t>0$ that
\begin{align*}
    \int_Z \fint_{B(x, t)} |f(y) - f(x)|\,d\nu(y) \,d\nu(x)
  &\le \int_Z \fint_{B(x, t)} \bigl(|f(y)| + |f(x)|\bigr)\,d\nu(y)\, d\nu(x) \\
  & = \int_Z \biggl(|f(x)| + \fint_{B(x, t)} |f(y)| \,d\nu(y)\biggr)d\nu(x)\\
  & = \|f\|_{L^1(Z)} +  \int_Z \fint_{B(x, t)} |f(y)| \,d\nu(y)\,d\nu(x).
\end{align*}
The Fubini theorem and the doubling condition then yield
\begin{align*}
  \int_Z \fint_{B(x, t)} |f(y)| \,d\nu(y)\,d\nu(x)
& \approx \int_{Z\times Z} \frac{|f(y)| \chi_{(0,t)}(\dd(x,y))}{\nu(B(y,t))}\, d(\nu\times\nu)(x,y) \\
  & = \int_Z |f(y)| \fint_{B(y, t)} \,d\nu(x)\,d\nu(y) = \|f\|_{L^1(Z)}\,.
\end{align*}
For $r>0$ set
\[
 I(r):=\int_0^r \int_Z \fint_{B(x, t)} |f(y) - f(x)|\,d\nu(y) \,d\nu(x) \frac{dt}{t^{1+\theta}}.
\]
Then,
\begin{align*}
I(R)&=I(r)+
  \int_r^R \int_Z \fint_{B(x, t)} |f(y) - f(x)|\,d\nu(y) \,d\nu(x) \frac{dt}{t^{1+\theta}}\\
  &\le I(r)+C\int_r^R \|f\|_{L^1(Z)} \frac{dt}{t^{1+\theta}}
  = I(r)+C \|f\|_{L^1(Z)} \Bigl(\frac{1}{r^{\theta}} - \frac{1}{R^\theta}\Bigr),
\end{align*}
with obvious modification for $\theta=0$.
\end{proof}
\begin{lemma}
\label{lem:BesovEquiv}
Let $R \le 2 \diam(Z)$ and $\theta \in [0, 1]$. Then,
\[
  I(R) \approx \int_Z \int_{B(x,R)} \frac{|f(y) - f(x)|}{ \nu(B(x, \dd(x,y))) \dd(x,y)^\theta}\,d\nu(y)\,d\nu(x) .
\]
\end{lemma}
\begin{proof}
The equivalence follows from the Fubini theorem, see
also~\cite[Theorem~5.2]{GKS}.
\end{proof}
\begin{lemma}
\label{lem:BesovEquivFixedBalls}
There is a constant $C\ge 1$ and there are collections of balls $\bcal^k$,
$k=0,1,\ldots$, such that
\[
  C^{-1}I(1)
 \le\sum_{k=0}^\infty \sum_{B\in\bcal^k} \frac{1}{\rad(B)^\theta}\int_B |f-f_B|\,d\nu
   \le C I(4).
\]
Moreover, $\rad(B) \approx 2^{-k}$ whenever $B\in\bcal^k$, and the balls within
each collection $\bcal^k$ have bounded overlap (also after inflation by a given
factor $\tau\ge1$).
\end{lemma}
\begin{proof}
By Lemma~\ref{cor:ballcovering}, there is a constant $C = C(C_D, \tau)\ge1$ and
collections of balls $\widetilde{\bcal}^k$, $k=0, 1,\ldots$, such that $\rad
(B) = 2^{-k}$ for every $B \in \widetilde{\bcal}^k$ and $1 \le \sum_{B \in
\widetilde{\bcal}^k} \chi_{2\tau B}(x) \le C$ for all $x\in Z$. Then,
by~\eqref{eq:BesovEquivSum},
\begin{align*}
  I(1)
& \le \sum_{k=0}^\infty 2^{k\theta} \sum_{B\in\widetilde{\bcal}^k} \int_B \fint_{B(x,2^{-k})} |f(y) - f(x)| \,d\nu(y)\,d\nu(x) \\
  & \le C \sum_{k=0}^\infty 2^{k\theta} \sum_{B\in\widetilde{\bcal}^k} \int_B \fint_{2B} |f(y) - f(x)| \,d\nu(y)\,d\nu(x) \\
  & \le C \sum_{k=0}^\infty \sum_{B\in\widetilde{\bcal}^k} \frac{1}{(2^{-k})^\theta}\int_{2B} \fint_{2B} |f(y) - f(x)| \,d\nu(y)\,d\nu(x)\\
  & \le C \sum_{k=0}^\infty \sum_{B\in\widetilde{\bcal}^k} \frac{1}{\rad(B)^\theta} \int_{2B} |f-f_{2B}|\,d\nu.
\end{align*}
Thus, we choose $\bcal^k = \{2B: B\in \widetilde{\bcal}^k\}$, $k=0, 1,\ldots$,
to conclude the proof of the first inequality.

The proof of the second inequality follows analogous steps backwards. Recall
that $\rad(B) = 2^{1-k}$ whenever $B\in\bcal^k$. Thus,
\begin{align*}
\sum_{k=0}^\infty \sum_{B\in\bcal^k}& \frac{1}{\rad(B)^\theta} \int_{B} |f-f_{B}|\,d\nu \\
& \approx \sum_{k=0}^\infty \sum_{B\in\bcal^k} 2^{k\theta} \int_B \fint_B |f(y) - f(x)|\,d\nu(y)\,d\nu(x) \\
& \le C \sum_{k=0}^\infty 2^{k\theta} \sum_{B\in\bcal^k} \int_B \fint_{B(x,2^{2-k})} |f(y) - f(x)|\,d\nu(y)\,d\nu(x) \\
\displaybreak[1] & \le C \sum_{k=0}^\infty 2^{k\theta} \int_Z \fint_{B(x,2^{2-k})} |f(y) - f(x)|\,d\nu(y)\,d\nu(x),
\end{align*}
where we used the fact that the balls have uniformly bounded overlap within
each collection $\bcal^k$.
\end{proof}
\begin{remark}
Let $0 \le \theta < \eta \le 1$. Then, $\|u\|_{B_{1,1}^\theta} \le C
(1+R^{\eta-\theta}) \|u\|_{B_{1,1}^\eta}$ and $ \|u\|_{A^{\theta}_{1,\tau}} \le
C (1+R^{\eta-\theta}) \|u\|_{A^{\eta}_{1,\tau}}$ for any $\tau \ge 1$.
\end{remark}
\subsection{Comparison of function spaces with $\Be^\theta(Z)$}
\begin{proposition}
Let $\theta \in [0,1]$ and $\tau \ge 1$ be arbitrary. Then there is a constant
$C\ge 1$, which depends on $\theta$ and $\tau$, such that
\[
\|u\|_{A^{\theta}_{1,\tau}} \le C \|u\|_{B_{1,1}^\theta}.
\]
\end{proposition}
\begin{proof}
Let $\bcal_\tau$ be a fixed collection of non-overlapping balls in $Z$ of
radius at most $R/\tau$. Then,
\begin{align*}
  \sum_{B\in\bcal_\tau} \frac{1}{\rad(B)^\theta} &\int_{\tau B}|u-u_{\tau B}|\,d\nu\\
& \approx \sum_{B\in\bcal_\tau} \frac{1}{(\tau\rad(B))^\theta} \int_{\tau B}\fint_{\tau B}|u(x)-u(y)|\,d\nu(y)\,d\nu(x) \\
  & = \int_Z \sum_{B\in\bcal_\tau} \frac{\chi_{\tau B}(x)}{(\tau\rad(B))^\theta}\fint_{\tau B}|u(x)-u(y)|\,d\nu(y)\,d\nu(x) \\
  & \le \int_Z \sum_{B\in\bcal_\tau} \chi_{\tau B}(x)\int_{\tau B}\frac{|u(x)-u(y)|}{\nu(B(x, \dd(x,y)))\dd(x,y)^\theta}\,d\nu(y)\,d\nu(x) \\
  & \le \int_Z \int_{B(x,2R)}\frac{|u(x)-u(y)|}{\nu(B(x, \dd(x,y)))\dd(x,y)^\theta}\,d\nu(y)\,d\nu(x) \approx I(2R),
\end{align*}
where we used that $\nu$ is doubling and $B(x, \dd(x,y)) \subset 3\tau B$ for
all $x,y\in\tau B$. Taking supremum over all collections of balls concludes the
proof.
\end{proof}
\begin{proposition}
Assume that $Z$ is bounded. Let $0 \le \theta < \eta \le 1$ and $\tau \ge 1$.
Then, $A_{1,\tau}^\eta(Z) \subset B_{1,1}^\theta(Z)$.
\end{proposition}
\begin{proof}
We use the characterization of Besov functions from
Lemma~\ref{lem:BesovEquivFixedBalls}.
\begin{align*}
   \sum_{k=0}^\infty &\sum_{B\in\bcal^k} \frac{1}{\rad(B)^\theta}\int_B |f-f_B|\,d\nu \\
  & \le C \sum_{k=0}^\infty
  \sum_{B\in\bcal^k} \frac{2^{k(\theta-\eta)}}{\rad(B)^\eta}\int_B |f-f_B|\,d\nu
  \le C \sum_{k=0}^\infty 2^{k(\theta-\eta)} \|f\|_{A_{1,\tau}^\eta(Z)}. \qedhere
\end{align*}
\end{proof}
\begin{lemma}
For every $\tau \ge 1$, we have $L^1(Z) = A_{1,\tau}^0(Z)$.
\end{lemma}
\begin{proof}
Let $\bcal_\tau$ be a fixed collection of balls in $Z$ that remain pairwise
disjoint after being inflated $\tau$-times. Then, the triangle inequality
yields that
\begin{align*}
  \sum_{B\in\bcal_\tau} \frac{1}{\rad(B)^0} \int_{\tau B}|u-u_{\tau B}|\,d\nu
  &\approx \sum_{B\in\bcal_\tau} \biggl(\int_{\tau B}|u| + \biggl|\fint_{\tau B} u\,d\nu\biggr|\biggr)\,d\nu\\
  &= \sum_{B\in\bcal_\tau} 2 \int_{\tau B}|u| \,d\nu \le 2  \int_{Z} |u| \,d\nu\,.
\end{align*}
Hence, $\|u\|_{A_{1,\tau}^0(Z)} \le 3 \|u\|_{L^1(Z)}$.

Conversely, $\|u\|_{A_{1,\tau}^0(Z)} \ge \|u\|_{L^1(Z)}$ by the definition
\eqref{eq:JN}.
\end{proof}
The following example shows that the inclusion $B_{1,1}^0(Z) \subset
A_{1,\tau}^0(Z)$ may, in general, be strict.
\begin{example}
\label{exa:Bes0!=L1}
Let
\[
  f(x) = \sum_{j=1}^\infty \chi_{[1/(j+1), 1/j)}(x) u(4^j x), \quad x\in(0,1),
\]
where $u$ is the $1$-periodic extension of $\chi_{[0,1/2)}$. Obviously, $f\in
L^\infty(0,1)$. Hence, $f \in L^1(0,1) = A_{1,\tau}^0(0,1)$. On the other hand,
$u \notin B_{1,1}^0(0,1)$, which we are about to show.

We will use the characterization of $B_{1,1}^0(0,1)$ from
Lemma~\ref{lem:BesovEquivFixedBalls}. There, we may choose $\bcal^k = \{ (l
2^{-k}, (l+2) 2^{-k}): l=0,1,\ldots, 2^k-2\}$ to get $\fint_B |f-f_B| \approx
1$ whenever $B \subset (0, 1/j)$ and $B\in\bcal^k$ for some $k\le j$. Then,
\[
  \sum_{k=0}^\infty \sum_{B\in\bcal^k} \int_{B} |f-f_B|
\ge C^{-1} \sum_{k=0}^\infty \sum_{\substack{B\in\bcal^k\\ B \subset (0,k^{-1})}} |B|
\approx \sum_{k=0}^\infty \frac{1}{k} = \infty.
\]
\end{example}
Next, we will provide a family of examples that show that $BV(Z)$ is, in
general, a strictly smaller space than $B_{1,1}^\theta(Z)$ for every
$\theta\in[0,1)$.
\begin{example}
Let $\alpha \in (\theta, 1)$. Then, the Weierstrass function
\[
  u_\alpha (x) = \sum_{k=1}^\infty \frac{\cos(2^k \pi x)}{2^{k\alpha}}\,, \quad x\in[0,1],
\]
is $\alpha$-H\"older continuous but nowhere differentiable in $[0,1]$ by
Hardy~\cite{Har}. Hence, $u_\alpha \notin BV[0,1]$ as it would have been
differentiable a.e.\@ otherwise. Since $\alpha > \theta$, we have
$\ccal^{0,\alpha}[0,1] \subset B_{1,1}^\theta[0,1]$ by~\cite[Lemma~6.2]{GKS}.
\end{example}
In conclusion, we have now proved the following theorem.
\begin{theorem}
Let $\tau\ge 1$ and $\theta \in (0,1]$ be arbitrary. Then,
\[
  L^1(Z) = A_{1,\tau}^0(Z) \supset B_{1,1}^0(Z) \supset A_{1,\tau}^\theta (Z) \supset A_{1,\tau}^1 (Z) \subset BV(Z),
\]
where all but the last of the inclusions are strict in general. Furthermore,
Lipschitz functions on $Z$ belong to $\Be^0(Z)$.
\end{theorem}
We know from~\cite[Theorem~1.1]{HKT} that $A_{1,\tau}^1(Z)\subset BV(Z)$. Note
however that $A_{1,\tau}^1 (Z) = BV(Z)$ holds by~\cite[Corollary~1.3]{HKT}
whenever $Z$ supports a $1$-Poincar\'e inequality.
\section{Extension theorem for $L^1$ boundary data: proof of Theorem~\ref{thm:main2}}\label{sec:l1-extension}
Given an $L^1$-function on $\dOm$, we will construct its BV extension in $\Om$
using the linear extension operator for $B_{1,1}^0(\dOm)$ boundary data.
Observe however that the mapping $f\in L^1(\dOm) \mapsto F\in BV(\Om)$ will be
nonlinear, which is not surprising in view of \cite{Pee}.

Instead of constructing the extension using a Whitney decomposition of
$\Omega$, we will set up a sequence of layers inside $\Om$ whose widths depend
not only on their distance from $X\setminus \Om$, but also on the function
itself (more accurately, on the choice of the sequence of Lipschitz
approximations of the function in $L^1$-class). Using a partition of unity
subordinate to these layers, we will glue together BV extensions (from
Theorem~\ref{thm:main1}) of Lipschitz functions on $\dOm$ that approximate the
boundary data in $L^1(\dOm)$. Roughly speaking, the closer the layer lies to
$X\setminus \Om$, the better we need the approximating Lipschitz data to be.
The core idea of such a construction can be traced back to Gagliardo~\cite{Gag}
who discussed extending $L^1(\RR^{n-1})$ functions to $W^{1,1}(\RR^n_+)$.

First, we approximate $f$ in $L^1(\dOm)$ by a sequence of Lipschitz continuous
functions $\{f_k\}_{k=1}^\infty$ such that $\|f_{k+1} - f_k\|_{L^1(\dOm)} \le
2^{2-k} \|f\|_{L^1(\dOm)}$. Note that this requirement of rate of convergence
of $f_k$ to $f$ also ensures that $f_k\to f$ pointwise $\mathcal{H}$-a.e.\@ in
$\dOm$. For technical reasons, we choose $f_1 \equiv 0$.

Next, we choose a decreasing sequence of real numbers $\{\rho_k\}_{k=1}^\infty$
such that:
\begin{itemize}
  \setlength{\parskip}{0pt}
  \setlength{\itemsep}{2pt plus 1pt minus 2pt}
	\item $\rho_1 \le \diam(\Omega)/2$;
  \item $0<\rho_{k+1} \le \rho_k /2$;
  \item $\sum_k \rho_k  \LIP(f_{k+1}, \dOm) \le C \|f\|_{L^1(\dOm)}$.
\end{itemize}
These will now be used to define layers in $\Omega$. Let
\[
  \psi_k (x) = \max\biggl\{ 0, \min\biggl\{1, \frac{\rho_k - \dist(x, X\setminus \Om)}{\rho_k - \rho_{k+1}} \biggr\} \biggr\},
  \quad x \in \Om.
\]
Then, the sequence of functions $\{\psi_{k-1} - \psi_k: k=2,3,\ldots\}$ serves
as a partition of unity in $\Om(0, \rho_2)$ subordinate to the system of layers
given by $\{ \Om(\rho_{k+1}, \rho_{k-1}): k=2,3,\ldots\}$.

Recall that Lipschitz continuous functions lie in the Besov class $B_{1,1}^0$.
Thus, we can apply the linear extension operator $E: B_{1,1}^0(\dOm) \to
BV(\Om)$, whose properties were established in
Section~\ref{sec:B110-extension}, to define the extension of $f \in L^1(\dOm)$
by extending its Lipschitz approximations in layers, i.e.,
\begin{align}
  \notag
  F(x) & := \sum_{k=2}^\infty \bigl(\psi_{k-1}(x) - \psi_{k}(x)\bigr) Ef_k(x) \\
  & \phantom{:}= \sum_{k=1}^\infty \psi_k(x) \bigl(Ef_{k+1}(x) - Ef_{k}(x)\bigr), \quad x\in\Om.
  \label{eq:L1-ext}
\end{align}

The following result shows that the above extension is in the class $BV(\Om)$
with appropriate norm bounds (see Remark~\ref{rem:lip-upp2}). Indeed, we will
see that the extension given below lies in $N^{1,1}(\Om)\subset BV(\Om)$.
\begin{proposition}\label{prop:trace-L1}
Given $f \in L^1(\dOm)$, the extension defined by \eqref{eq:L1-ext} satisfies
\begin{align*}
  \| F\|_{L^1(\Omega)} & \le C \diam(\Omega) \|f\|_{L^1(\dOm)} \quad\mbox{and}\\
  \|\Lip F\|_{L^1(\Omega)} & \le C (1+\hcal(\partial\Om)) \|f\|_{L^1(\dOm)}.
\end{align*}
\end{proposition}
\begin{proof}
Corollary~\ref{layer-est-Fn} allows us to obtain the desired $L^1$ estimate for
$F$. Since the extension on $\Be^0(\dOm)$ is linear, we have that
$Ef_{k+1}-Ef_{k}=E(f_{k+1}-f_{k})$. Therefore,
\begin{align*}
  \| F\|_{L^1(\Omega)}  &\le \sum_{k=1}^\infty \| \psi_k E(f_{k+1}-f_{k})\|_{L^1(\Omega)}
   \le \sum_{k=1}^\infty \| E(f_{k+1}-f_{k})\|_{L^1(\Omega(0, \rho_{k}))} \\
   \le &C \sum_{k=1}^\infty \rho_{k} \|f_{k+1}-f_{k}\|_{L^1(\dOm)} \le C \rho_1 \|f\|_{L^1(\dOm)}
    \le C \diam(\Om)  \|f\|_{L^1(\dOm)}\,.
\end{align*}
In order to obtain the $L^1$ estimate  for $\Lip F$, we first apply the product
rule for locally Lipschitz functions, which yields that
\begin{align*}
  \Lip F & = \sum_{k=1}^\infty\bigl( |E(f_{k+1}-f_{k})|\Lip\psi_k + \psi_k \Lip(E(f_{k+1}-f_{k})) \bigr)\\
  & \le \sum_{k=1}^\infty \biggl(\frac{|E(f_{k+1}-f_{k})| \chi_{\Omega(\rho_{k+1}, \rho_k)}}{\rho_k-\rho_{k+1}}
  + \chi_{\Omega(0, \rho_k)} \Lip(E(f_{k+1}-f_{k}))\biggr)
\end{align*}
It follows from Corollary~\ref{layer-est-Fn} that
\begin{align*}
\sum_{k=1}^\infty \biggl\| \frac{E(f_{k+1}-f_{k})}{\rho_k-\rho_{k+1}} \biggr\|_{L^1(\Omega(\rho_{k+1}, \rho_k))}
& \le C \sum_{k=1}^\infty  \frac{\rho_k}{\rho_k-\rho_{k+1}} \|f_{k+1}-f_{k}\|_{L^1(\dOm)} \\
& \le C \sum_{k=1}^\infty  \|f_{k+1}-f_{k}\|_{L^1(\dOm)} \le C \|f\|_{L^1(\dOm)}\,.
\end{align*}
Next, we apply Corollary~\ref{cor:layer-est-grad-lip} to see that
\begin{align*}
  \sum_{k=1}^\infty \bigl\| \Lip E(f_{k+1}-f_{k}) \bigr\|_{L^1(\Omega(0, \rho_k))}
    \le C & \sum_{k=1}^\infty \rho_k \hcal(\dOm) \LIP (f_{k+1}-f_{k},\dOm) \\
  \le C \hcal(\dOm) & \sum_{k=1}^\infty \rho_k \bigl(\LIP (f_{k+1},\dOm)+\LIP(f_{k},\dOm)\bigr) \\
  \le C \hcal(\dOm) & \|f\|_{L^1(\dOm)},
\end{align*}
where we used the defining properties of $\{\rho_k\}_{k=1}^\infty$ to obtain
the ultimate inequality. Altogether, $\|\Lip F\|_{L^1(\Om)} \le C (1+
\hcal(\dOm)) \|f\|_{L^1(\dOm)}$.
\end{proof}
\subsection{Trace of the extended functions}\label{subsec-L1Trace}
In this section we complete the proof of Theorem~\ref{thm:main2} by showing
that the trace of the extended function yields the original function back.
\begin{proposition}
Let $F \in BV(\Om)$ be the extension of $f\in L^1(\dOm)$ as constructed in
\eqref{eq:L1-ext}. Then,
\[
  \lim_{r\to 0} \fint_{B(z,r)\cap \Om} |F - f(z)| \,d\mu = 0
\]
for $\hcal$-a.e.\@ $z\in \dOm$.
\end{proposition}
\begin{proof}
Let $E_0$ be the collection of all $z\in\dOm$ for which $\lim_kf_k(z)=f(z)$,
and for $k\in\NN$ let $E_k$ be the collection of all $z\in\dOm$ for which
$TEf_k(z)=f_k(z)$ exists. Lemma~\ref{trace-Besov} yields that $\hcal(\dOm
\setminus \bigcap_{k=0}^\infty E_k) = 0$. We define also an auxiliary sequence
$\{F_n\}_{n=1}^\infty$ of functions approximating $F$ by
\[
  F_n = \sum_{k=2}^n (\psi_{k-1}-\psi_k) E{f_k} + \sum_{k=n+1}^\infty (\psi_{k-1}-\psi_k) E{f_n}, \quad n\in\NN.
\]
It can be shown that $F_n\to F$ in $BV(\Om)$, but we will not need this fact
here. Note that $F_n = Ef_n$ in $\Om(0, \rho_n)$ and hence the trace of $F_n$
exists on $\dOm$ and coincides with the trace of $Ef_n$, i.e., with $f_n$.

Fix a point $z \in \bigcap_{k=0}^\infty E_k$ and let $\eps>0$. Then, we can
find $j \in \NN$ such that $|f_k(z) - f(z)| < \eps$ for every $k\ge j$. Next,
we choose $k_0>j$ such that $R := \rho_{k_0}$ satisfies:
\begin{itemize}
  \setlength{\parskip}{0pt}
  \setlength{\itemsep}{2pt plus 1pt minus 2pt}
	\item $R \LIP(f_j, \dOm) < \eps$;
  \item $\fint_{B(z,r)} |F_j - f_j(z)|\,d\mu < \eps$ for every $r < R$;
  \item $\sum_{k=k_0}^\infty \rho_k \LIP(f_{k+1}, \dOm) < \eps$.
\end{itemize}
For every $r \in (0, \rho_{k_0+1}) \subset (0, R/2)$, we can then estimate
\begin{align}
  \notag
  \fint_{B(z,r)\cap\Om} &|F-f(z)|\,d\mu \\
  \notag
  & \le \fint_{B(z,r)\cap\Om} |F-F_j|\,d\mu + \fint_{B(z,r)\cap\Om} |F_j - f_j(z)|\,d\mu + |f_j(z) - f(z)| \\
  \label{eq:L1-trace}
  & \le \fint_{B(z,r)\cap\Om} |F-F_j|\,d\mu + 2\eps.
\end{align}
For such $r$, choose $k_r> k_0$ such that $\rho_{k_r+1} \le r < \rho_{k_r}$.
Then,
\begin{align}
  \notag
  \int_{B(z,r)\cap \Om} |F-F_j|\,d\mu & \le \sum_{k=k_r}^\infty \int_{B(z,r)\cap\Om} (\psi_{k-1} - \psi_k) \bigl|E(f_k - f_j)\bigr|\,d\mu \\
  \notag
 & \le \sum_{k=k_r}^\infty \int_{B(z,r) \cap \Om(\rho_{k+1}, \rho_{k-1})} \bigl|E(f_k - f_j)\bigr|\,d\mu \\
 \label{eq:L1-trace-A}
 & \le C \sum_{k=k_r}^\infty \min\{r,\rho_{k-1}\} \int_{B(z,2^8 r) \cap \dOm} |f_k - f_j|\,d\hcal
\end{align}
by Lemma~\ref{layer-est-Fn-boundaryball}. In the last inequality above, we used
the fact that when $k=k_r$, we must have $B(z,r)\cap \Om(\rho_{k_r+1},
\rho_{k_r-1}) = B(z,r)\cap\Om(\rho_{k_r+1},r)$ by the choice of $r<\rho_{k_r}$.

Let us, for the sake of brevity, write $U_r = B(z, 2^8 r) \cap \dOm$. As $f_k -
f_j$ is Lipschitz continuous, we have by the choice of $j$, and the fact that
$k\ge j$,
\begin{align}
  \notag
  \int_{U_r} |f_k - f_j|\,d\hcal
  & \le \int_{U_r} \bigl|f_k - f_j - (f_k(z) - f_j(z))\bigr|\,d\hcal + |f_k(z) - f_j(z)| \hcal(U_r) \\
 \label{eq:L1-trace-B}
  & \le Cr \hcal(U_r) \LIP(f_k - f_j, U_r)  + 2 \eps \hcal(U_r).
\end{align}
Observe that $r \hcal(U_r) \approx \mu(B(z, r))$ by
\eqref{boundary-Ahlfors-regularity}, and the doubling condition for $\mu$. Note
that $\sum_{k=k_r}^\infty\rho_{k-1}\le C\rho_{k_r-1}\le CR$. Combining this
with~\eqref{eq:L1-trace-A} and \eqref{eq:L1-trace-B} gives us that
\begin{align*}
  \int_{B(z,r)\cap \Om} |F& -F_j|\,d\mu \le \sum_{k=k_r}^\infty C \rho_{k-1} \mu(B(z,r)) \bigl(\LIP(f_k, \dOm) + \LIP(f_j, \dOm)\bigr) \\
 & \phantom{\le} + 2 \eps \mu(B(z,r)) \sum_{k=k_r}^\infty \frac{\min\{r, \rho_{k-1}\}}{r} \\
 &\le C \mu(B(z,r)) \biggl(\sum_{k=k_0}^\infty \bigl(\rho_k \LIP(f_{k+1}, \dOm)\bigr) +  R \LIP(f_{j}, \dOm) + \eps\biggr)\\
 & \le C \mu(B(z,r)) \eps.
\end{align*}
Plugging this estimate into~\eqref{eq:L1-trace} completes the proof.
\end{proof}

\subsection{Summary and further discussion}
In conclusion, we have shown that every function in $L^1(\dOm)$ has an
extension to $BV(\Om)$ in such a way that the trace of the extension returns
the original function. This extension is nonlinear, but it is bounded. In a
preceding section we demonstrated that there is a bounded linear extension from
the subclass $\Be^0(\dOm)$ to $BV(\Om)$. Note that $\Be^0(\dOm)$, containing
all the Lipschitz functions on $\dOm$, must necessarily be dense in
$L^1(\dOm)$. It therefore follows that this extension from $L^1(\dOm)$ to
$BV(\Om)$ cannot be continuous on $L^1(\dOm)$ since if it were, then the
extension from $L^1(\dOm)$ would be bounded and linear--and this is not
possible (see~\cite{Pee} for the fact that in general any extension from
$L^1(\dOm)$ to $BV(\Om)$ cannot be both bounded and linear).

On the other hand, in the setting of Corollary~\ref{cor:TrBV=L1}, using this
corollary we see that the trace operator $T$ of~\cite{LS} is a continuous
\emph{surjective} linear mapping of the Banach space $BV(\Om)$ to the Banach
space $L^1(\dOm)$, and hence there exists a continuous (non-linear) right
inverse of $T$ by the Bartle--Graves theorem~\cite[Corollary~7.1]{BP}. Should
we know that $\Om$ is a domain for which each function $f\in L^1(\partial\Om)$
has an associated function $u_f\in BV(\Om)$ such that $u_f$ is of least
gradient (that is, $1$-harmonic) in $\Om$ and with trace $Tu=f$
$\mathcal{H}$-a.e.~in $\partial\Om$, then the natural continuous inverse map
would be the map $f\mapsto u_f$; the stability results of~\cite{HKLS} would
indicate that this map is continuous. However, even under the best of
circumstances, for example $\Om$ the unit disk in $\R^2$, no $u_f$ exists for
general $f\in L^1(\partial\Om)$ (see~\cite{ST}). It is not clear what the
Bartle-Graves inverse map is.
\noindent Address:

\noindent L.\@ Mal\'{y}:
Department of Mathematical Sciences,
         University of Cincinnati,
         Cincinnati, OH 45221-0025,
         U.S.A.\\
\noindent {\tt malyls@ucmail.uc.edu}

\vskip 6pt plus 2pt minus 6pt
\noindent N.\@ Shanmugalingam:
Department of Mathematical Sciences,
         University of Cincinnati,
         Cincinnati, OH 45221-0025,
         U.S.A.\\
\noindent {\tt shanmun@ucmail.uc.edu}

\vskip 6pt plus 2pt minus 6pt
\noindent M.\@ Snipes:
Department of Mathematics,
         Kenyon College,
         Gambier, OH~43022,
         U.S.A.\\
\noindent {\tt snipesm@kenyon.edu}
\end{document}